\documentclass[11pt,reqno]{amsart}
\usepackage[T1]{fontenc}
\usepackage{newtxtext,newtxmath}
\usepackage[margin=1in]{geometry}
\usepackage{microtype}
\usepackage{enumitem}
\usepackage{needspace}
\usepackage{xcolor}
\usepackage[colorlinks=true,linkcolor=black,citecolor=black,urlcolor=blue!45!black]{hyperref}
\hypersetup{pdftitle={A finitely generated linear group without property (T) that is neither inner amenable nor properly proximal},
  pdfsubject={A finitely generated linear group that is not inner amenable and not properly proximal}}
\newtheorem{theorem}{Theorem}[section]
\newtheorem{lemma}[theorem]{Lemma}
\newtheorem{proposition}[theorem]{Proposition}

\theoremstyle{definition}

\theoremstyle{remark}
\newtheorem{remark}[theorem]{Remark}
\newcommand{\F}{\mathbb F}
\newcommand{\Z}{\mathbb Z}
\newcommand{\C}{\mathbb C}
\newcommand{\SL}{\mathrm{SL}}
\newcommand{\GL}{\mathrm{GL}}

\newcommand{\one}{\mathbf 1}

\newcommand{\FC}{\operatorname{FC}}

\setlist[enumerate]{label=\textup{(\roman*)},leftmargin=2.2em,itemsep=2pt}
\title[A finitely generated extension of Caprace's example]{A finitely generated linear group without property (T) that is neither inner amenable nor properly proximal}

\author{Xin Ma}

	\address{X. Ma:  Institute for Advanced Study in Mathematics, Harbin Institute of Technology, Harbin, China, 150001}
  \email{xma17@hit.edu.cn}

\date{\today}

\begin{document}
\begin{abstract}
We construct an explicit finitely generated ICC subgroup of $\GL_4(\F_2(t))$ that is neither inner amenable nor properly proximal and does not have Kazhdan's property~\textup{(T)}. The group is obtained by adjoining a scalar dilation to Caprace's example. In particular, this has resolved a question due to Boutonnet, Ioana, and Peterson in negative.
\end{abstract}
\maketitle

\section{Introduction}
Amenability and its
variants connect invariant means on groups with the
structure of their von Neumann algebras. A \emph{mean} is a finitely additive
probability measure defined on all subsets of a set. A discrete group is
\emph{amenable} if it admits a left-invariant mean, and it is
\emph{inner amenable} if it admits a conjugation-invariant mean that vanishes
on singletons. Inner amenability originates in Effros's work on property
$\Gamma$ \cite{Eff75}. For an \emph{ICC} group, meaning that every
nonidentity conjugacy class is infinite, property $\Gamma$ of the group
von Neumann algebra implies inner amenability. However, the converse fails by a counterexample 
constructed by Vaes in \cite{Vae12}. Here property $\Gamma$ means the existence
of trace-zero unitaries that asymptotically commute, in the tracial $2$-norm
$\|\cdot\|_2$, with every element of the factor.

Another relevant rigidity property is Kazhdan's property~\textup{(T)}, which was
introduced in 1967. A discrete group is said to have property~\textup{(T)} if every
unitary representation with almost invariant unit vectors has a nonzero
invariant vector \cite{BHV08}. Discrete groups with property~\textup{(T)} are necessarily finitely
generated, and the property passes to quotients. 

In \cite{BIP21}, Boutonnet, Ioana, and Peterson introduced \emph{proper proximality} to
extend boundary methods in the study of von Neumann algebras beyond
bi-exact groups. Their class includes nonamenable bi-exact
groups, non-elementary convergence groups, and higher-rank lattices. For a
countably infinite group $\Gamma$, an equivalent definition of proper proximality is that the algebra
\begin{equation*}\label{eq:pp-corona}
 \mathcal A_\Gamma=\bigl(\ell^\infty(\Gamma)/c_0(\Gamma)\bigr)^{\Gamma_r}=C(\beta\Gamma\setminus\Gamma)^{\Gamma_r}
\end{equation*}
admits no left $\Gamma$-invariant state \cite[Theorem~4.3]{BIP21}.
Here $\Gamma_r$
denotes right translation. Thus $\mathcal A_\Gamma$ consists of classes $[f]$ satisfying
$f(\,\cdot\,g)-f\in c_0(\Gamma)$ for every $g\in \Gamma$.
This formulation makes the obstruction from inner amenability transparent. Indeed,
an atomless conjugation-invariant mean $m: \ell^\infty(\Gamma)\to \C$ descends to the quotient $\bar{m}:\ell^\infty(\Gamma)/c_0(\Gamma)\to \C $ and
restricts to a left-invariant state on $\mathcal A_\Gamma$.
Consequently, every properly proximal group is not inner amenable
\cite[Proposition~4.11]{BIP21}. Moreover, in  \cite[Question 1(c)]{BIP21}, the
authors explicitly asked whether the converse is true and suggested that it might
hold  for finitely generated linear groups.

The general converse was disproved using measure equivalence in \cite{IPR24}. Two
countable groups are \emph{measure equivalent} if they admit commuting,
essentially free, measure-preserving actions on a common nonzero standard
$\sigma$-finite measure space, with finite-measure fundamental domains
for both actions. Ishan, Peterson, and Ruth proved that proper proximality
is invariant under measure equivalence and applied this to a construction
of Caprace \cite{IPR24}.  Denote by $\F_2=\{0,1\}$ the finite field with $2$ elements. As recorded by Duchesne, Tucker-Drob, and Wesolek
\cite[Section~5.C]{DTDW20}, the group
\begin{equation}\label{eq:caprace}
 \Gamma=\F_2[t,t^{-1}]^3\rtimes\SL_3(\F_2[t^{-1}])
\end{equation}
is not inner amenable but is measure equivalent to a group with infinite
center. The latter group is inner amenable, so $\Gamma$ is not properly
proximal \cite[Theorem~1.3 and the subsequent discussion]{IPR24}.
However, $\Gamma$ is not finitely generated (see
Remark~\ref{rem:not-fg}).
The purpose of the current paper is to give an explicit four-dimensional
construction in characteristic two by adjoining a scalar dilation to
Caprace's group.  The
resulting group is finitely generated and does not have
property~\textup{(T)} but remains
not inner amenable.  This thus resolve \cite[Question 1(c)]{BIP21} in negative.

\begin{theorem}\label{thm:main}
Let $t$ be an indeterminate, and define
\[
 R=\F_2[t^{-1}],\qquad L=\F_2[t,t^{-1}],\qquad K=\F_2(t),
 \qquad Q=\SL_3(R).
\]
Then the group
\begin{equation*}\label{eq:G}
 G=\left\{
 \begin{pmatrix}t^kA&v\\0&1\end{pmatrix}
 : A\in Q,\ v\in L^3,\ k\in\Z
 \right\}\leq\GL_4(K)
\end{equation*}
is finitely generated and ICC. Moreover, $G$ has no property~\textup{(T)}, and is neither inner amenable nor properly proximal.
\end{theorem}

\begin{remark}
Independently, Wang \cite{W} constructed ICC subgroups of
$\SL_7(\mathbb Q)$ with property~\textup{(T)} that are not properly proximal.
\end{remark}

\section{Finite generation of $G$}
In this section, we show that $G$ is finitely generated. For $v\in L^3$ and $A\in Q$, write
\[
 \tau_v=\begin{pmatrix}I_3&v\\0&1\end{pmatrix},\qquad
 d(A)=\begin{pmatrix}A&0\\0&1\end{pmatrix},\qquad
 s=\begin{pmatrix}tI_3&0\\0&1\end{pmatrix}.
\]
Every element of $G$ has the form $\tau_vd(A)s^k$. Matrix multiplication gives
\begin{equation}\label{eq:relations}
 s\tau_vs^{-1}=\tau_{tv},\qquad
 d(A)\tau_vd(A)^{-1}=\tau_{Av},\qquad
 sd(A)=d(A)s.
\end{equation}
These relations verify that $G$ is closed under products and inverses. 
The exponent $k$ is unique in the form $\tau_vd(A)s^k$ because the determinant of the upper-left block is $t^{3k}$. This allows to define a map
\[\pi:\begin{pmatrix}t^kA&v\\0&1\end{pmatrix}\mapsto k,
\] 
which can be verified as a well-defined surjective homomorphism with the kernel $\Gamma$. Therefore, one obtains a split exact sequence
\begin{equation}\label{eq:extension}
 1\longrightarrow\Gamma\longrightarrow G\longrightarrow\Z\longrightarrow1,
 \qquad G\cong\Gamma\rtimes\langle s\rangle.
\end{equation}

Denote by $e_1=(1, 0, 0)^T$, $e_2=(0, 1,0)^T$, and $e_3=(0, 0, 1)^T$ in $L^3$ for simplicity. 

\begin{proposition}\label{prop:fg}
The group $G$ is finitely generated.
\end{proposition}

\begin{proof}
Let $E_{ij}$ denote the matrix unit and define $x_{ij}(r)=I_3+rE_{ij}$ for $i\ne j$ and $r\in R$. We claim that the twelve elementary matrices
\[
 \{x_{ij}(1),x_{ij}(t^{-1}):1\leq i,j\leq3,\ i\ne j\}
\]
generate $Q$. Indeed, for distinct indices $i,j,\ell$, the identities
\[
 [x_{i\ell}(a),x_{\ell j}(b)]=x_{ij}(ab),\qquad
 x_{ij}(a)x_{ij}(b)=x_{ij}(a+b)
\]
produce every $x_{ij}(r)$ with $r\in R$. On the other hand, since $R$ is a Euclidean domain, these elementary matrices generate $\SL_3(R)$.

Add $s$ and $\tau_{e_1},\tau_{e_2},\tau_{e_3}$ to the embedded generators of $d(Q)$ in $G$. By \eqref{eq:relations}, the subgroup $\langle s, \tau_{e_1},\tau_{e_2}, \tau_{e_3}\rangle$ contains every translation $\tau_{t^ke_i}$, for $k\in\Z$. In addition, note that
\[\tau_{t^ke_i}\cdot \tau_{t^le_j}=\begin{pmatrix}I_3&t^ke_i\\0&1\end{pmatrix}\cdot\begin{pmatrix}I_3&t^le_j\\0&1\end{pmatrix}=\begin{pmatrix}I_3&t^ke_i+t^le_j\\0&1\end{pmatrix},\]
which implies
vectors $\tau_{t^ke_i}$ would produce all elements $\tau_v$ for $v\in L^3$.
Recall that every elements in $G$ is of the form $\tau_vd(A)s^k$ for a $v\in L^3$ and $A\in Q$. Thus, these sixteen elements generate the whole $G$.
\end{proof}

\begin{remark}\label{rem:not-fg}
The original group $\Gamma$ is the increasing union
\[
 \Gamma=\bigcup_{N\geq0}\bigl(t^NR^3\rtimes Q\bigr).
\]
Each term is a subgroup, because both $A$ and $A^{-1}$ preserve $t^NR^3$ for $A\in Q$, and each is proper: it omits $\tau_{t^{N+1}e_1}$. Every finite subset of $\Gamma$ lies in one term, so $\Gamma$ is not finitely generated. 
\end{remark}

\section{Proof of the main result}
In this section, we prove the main theorem.
A mean on a set $X$ is a positive unital linear functional on $\ell^\infty(X,\mathbb R)$ and we write $m(E)=m(\one_E)$ for $E\subseteq X$. In this note, \emph{atomless} means that every singleton has measure zero. Recall that  a group is inner amenable if it admits an atomless mean invariant under conjugation. We will start with the following consequence of property~\textup{(T)}, for which we include the proof.

\begin{lemma}\label{lem:T}
Let $H$ be a discrete group with property~\textup{(T)}, acting on a set $X$. Set
\[
 X_{\mathrm{fin}}=\{x\in X:|H\cdot x|<\infty\}.
\]
Every $H$-invariant mean $m$ on $X$ satisfies $m(X_{\mathrm{fin}})=1$.
\end{lemma}

\begin{proof}
Suppose $Y=X\setminus X_{\mathrm{fin}}$ has positive $m$-measure. Restricting and normalizing $m$ gives an $H$-invariant mean $\mu$ on $Y$.

We first claim that for every finite $F\subseteq H$ and $\varepsilon>0$, there exists a vector $p\in \ell^1(Y)$ with norm $1$ such that $p(x)\geq 0$ for any $x\in Y$ and 
\[
 \|h\cdot p-p\|_1\leq \varepsilon
\]
for every $h\in F$. To prove this, consider the real Banach space
$E=\bigoplus_{h\in F}\ell^1(Y,\mathbb R)$ with the $\ell^1$ norm.
Let $\mathcal P_f(Y)$ denote the finitely supported probability vectors on $Y$, and define
\[
 C=\{(h\cdot p-p)_{h\in F}:p\in\mathcal P_f(Y)\},
\]
which is a convex set. Suppose $0\notin \overline{C}$. Then the Hahn--Banach separation theorem implies that there exists a $\Phi\in E^*$ and $c>0$ such that $\Phi((h\cdot p-p)_{h\in F})>c$ holds for every $p\in \mathcal P_f(Y)$. This implies that there exists a $f_h\in \ell^\infty(Y)$ for $h\in F$ such that 
\[\Phi((h\cdot p-p)_{h\in F})=\sum_{h\in F}\sum_{y\in Y}f_h(y)(h\cdot p-p)(y).\]
Now, choose $p=\delta_y$ for a $y\in Y$. Since $h\cdot \delta_y=\delta_{hy}$, one actually has
\[\sum_{h\in F}(f_h(hy)-f_h(y))\geq c.\]
On the other hand, because $f_h(h y)=(h^{-1} \cdot f_h)(y)$, one has 
\[\mu\!\left(y\mapsto\sum_{h\in F}\bigl(f_h(hy)-f_h(y)\bigr)\right)=0.\]
Positivity therefore gives $0\geq c>0$, a contradiction. Hence $0\in\overline C$, which provides the required finitely supported probability vector $p$.

Taking pointwise square roots gives unit vectors $\xi=\sqrt{p}\in\ell^2(Y)$ with
\[
 \|h\cdot\xi-\xi\|_2^2\leq\|h\cdot p-p\|_1\leq \varepsilon.
\]
Property~\textup{(T)} then provides a nonzero invariant vector $\eta$ in $\ell^2(Y)$, i.e. $h\cdot \eta=\eta$ for all $h\in H$. This implies that $\eta(h^{-1}x)=\eta(x)$ for $h\in H$, i.e. $\eta$ is constant on each orbit. However, any orbit on which that constant is nonzero must be finite. This contradicts the definition of $Y$.
\end{proof}

We identify $Q$ with $d(Q)\leq G$ and compute the elements having finite orbit under conjugation by $Q$.

\begin{lemma}\label{lem:centralizer}
One has
\[
 \FC_G(Q):=\{g\in G:|\{qgq^{-1}:q\in Q\}|<\infty\}
 =\langle s\rangle.
\]
\end{lemma}

\begin{proof}
The inclusion $\langle s\rangle\subseteq\FC_G(Q)$ follows from \eqref{eq:relations}. Conversely, let
\[
 g=\begin{pmatrix}M&v\\0&1\end{pmatrix},\qquad M=t^kA,
\]
have finite $Q$-orbit. Then the stabilizer group $C_Q(g)$ for the conjugation action $Q\curvearrowright G$ has finite index in $Q$. For each $i\ne j$, the subgroup
\[
 U_{ij}=\{x_{ij}(r):r\in R\}
\]
is infinite, and $C_Q(g)\cap U_{ij}$ has finite index in $U_{ij}$ and therefore $C_Q(g)\cap U_{ij}$ is infinite. Thus it contains $x_{ij}(r_{ij})$ for some nonzero $r_{ij}\in R$ and thus $d(x_{ij}(r_{ij}))$ commutes with $g$. The explicit calculation implies that
\[r_{ij}E_{ij}M=r_{ij}ME_{ij}, \qquad r_{ij}E_{ij}v=0 \qquad (i\ne j).\]
Since we are working in the field $K$ and $r_{ij}\neq 0$, one actually has
\[
 E_{ij}M=ME_{ij},\qquad E_{ij}v=0\qquad(i\ne j)
\]
by dividing $r_{ij}$. The further calculation implies that $M=\lambda I_3$ for some $\lambda\in K^{\times}$ and $v=0$. Then because $M=t^kA$, one has $A=t^{-k}M=cI_3\in \operatorname{SL}_3(R)$.
The condition $\det(A)=1$ gives $c^3=1$, so $c\in R^\times=\{1\}$. Therefore $g=s^k$.
\end{proof}

We note that $Q=\SL_3(\F_2[t^{-1}])$ has Kazhdan's property~\textup{(T)} by \cite[Example 3.4.1]{BHV08}. This follows from the lattice realization in $\SL_3(\F_2((t)))$ and the standard property \textup{(T)} theorems for higher-rank groups and their lattices (see \cite[Theorems~1.4.15 and~1.7.1]{BHV08}). Note that this fact is also used in \cite[Section~5.C]{DTDW20}.

\begin{proposition}\label{prop:not-inner}
Every conjugation-invariant mean on $G$ is the point mass at the identity. In particular, $G$ is not inner amenable.
\end{proposition}

\begin{proof}
Let $m$ be a conjugation-invariant mean on $G$ and denote by $S=\langle s\rangle$. Recall that $Q=\SL_3(\F_2[t^{-1}])$ has the property~\textup{(T)}. Applying Lemma~\ref{lem:T} to $Q\curvearrowright G$ by conjugation,  Lemma~\ref{lem:centralizer} implies $m(S)=1$.

Write $\tau=\tau_{e_1}$ for simplicity. Direct multiplication gives
\[
 \tau s^k\tau^{-1}
 =\begin{pmatrix}t^kI_3&(1-t^k)e_1\\0&1\end{pmatrix}.
\]
Since $t$ is an indeterminate, $(1-t^k)e_1=0$ if and only if $k=0$. Thus, one has
\[
 S\cap\tau S\tau^{-1}=\{1\}.
\]
Both sets on the left have $m$-measure one. This implies that
\[
 2=m(S)+m(\tau S\tau^{-1})
  =m(S\cup\tau S\tau^{-1})+m(\{1\})
  \leq 1+m(\{1\}).
\]
This implies that $m(\{1\})\geq 1$ and thus necessarily one has $m(\{1\})=1$. So $m$ is the point mass at the identity and cannot be atomless.
\end{proof}

\begin{proposition}\label{prop: ICC}
    The group $G$ is ICC.
\end{proposition}
\begin{proof}
    Suppose there exists a $g\in G\setminus\{1_G\}$ whose conjugacy class $C$ is finite. Then define a conjugation-invariant mean $m_C$ by
    \[m_C(f)=(1/|C|)\sum_{x\in C}f(x).\]
    Note that $1_G\notin C$. Therefore, $m_C$ is not the point mass at the identity. This is a contradiction to Proposition \ref{prop:not-inner}.
\end{proof}

We are now able to prove Theorem \ref{thm:main}.  Recall that $\Lambda\leq H$ is \emph{co-amenable} if the left action of $H$ on $H/\Lambda$ admits an invariant mean. In particular, a normal subgroup with amenable quotient is co-amenable.

\begin{proof}[Proof of Theorem~\ref{thm:main}]
  Finite generation, non-inner-amenability, and the property of ICC, for $G$ have been established in Propositions \ref{prop:fg}, \ref{prop:not-inner} and \ref{prop: ICC}. Then note that the normal subgroup $\Gamma$ is co-amenable in $G$ because $G/\Gamma$ is isomorphic to $\Z$. If $G$ is properly proximal and $\Gamma\leq G$ is co-amenable, then $\Gamma$ is properly proximal \cite[Proposition~4.10(2)]{BIP21}, which is a contradiction to
 that $\Gamma$ in \eqref{eq:caprace} is not properly proximal by \cite[discussion following Theorem~1.3]{IPR24}. Thus $G$ is not properly proximal. Moreover, since $G$ has a quotient isomorphic to $\Z$, it is not of property \textup{(T)}.
\end{proof}

\section*{Acknowledgements}
The author is supported by NSFC No.~12571133. 
The main result of the paper was obtained in Sep 21, 2026. During the development of this work, the authors used GPT 6 Astra for searching necessary references, examining mathematical arguments, and language editing.  The author has checked and rewritten the proofs generated by the AI and takes full responsibility for the final version. 
The author thanks Sayan Das and Changying Ding for helpful comments and Ning Ma for discussions on linear groups.

\end{document}